\documentclass[11pt, letterpaper]{amsart}
\usepackage[utf8]{inputenc}
\usepackage{enumerate, xcolor, graphicx}
\definecolor{Magenta}{cmyk}{0,1,0,0}
\definecolor{dgreen}{rgb}{0,0.6,0.}

\usepackage{academicons}

\usepackage{amssymb,amsthm,amsmath,amsfonts}
\usepackage{diagbox}

\usepackage{mathrsfs, mathtools, mathdots}
\usepackage{paralist,fancyhdr,etoolbox,nicefrac,mathpazo}
\usepackage{listings}%for including sagecode
\usepackage{tikz-cd}%diagrampictures
\usepackage{parskip,xfrac}
\usepackage{datetime}
\usepackage{hyperref}
\usepackage{cleveref}
\usepackage{etoolbox}
\usepackage[top=27mm,bottom=27mm,left=25mm,right=25mm]{geometry}
\hypersetup{
    colorlinks=true,
    citecolor=red,       % The "angry grader" red pen for citations
    linkcolor=darkblue,  % Internal cross-references (like \cref) in blue
    urlcolor=red         % URLs also in red
}
\usepackage[T1]{fontenc}
\usepackage[default,scale=1.5]{comicneue} % 'default' sets it as the main document font
\usepackage[italic]{mathastext} % Forces math to use the Comic Neue font
\DeclareFontSeriesDefault[rm]{md}{bf}
\DeclareFontSeriesDefault[sf]{md}{bf}

\AtBeginDocument{\bfseries\boldmath}

\let\oldnormalfont\normalfont
\renewcommand{\normalfont}{\oldnormalfont\bfseries}
\boldmath
\usepackage[normalem]{ulem}     % The underline package
\newcommand{\ultrabold}{\fontseries{b}\selectfont}
\definecolor{darkblue}{rgb}{0.0, 0.0, 0.55} 

\makeatletter
\def\section{\@startsection{section}{1}%
  \z@{.7\linespacing\@plus\linespacing}{.5\linespacing}%
  {\normalfont\Large\bfseries\itshape\ultrabold\color{darkblue}}}
  
\def\subsection{\@startsection{subsection}{2}%
  \z@{.5\linespacing\@plus.7\linespacing}{-.5em}%
  {\normalfont\large\bfseries\itshape\ultrabold\color{darkblue}}}
\makeatother

\makeatletter
\let\oldsection\section
\renewcommand{\section}{\@ifstar{\@starredsection}{\@unstarredsection}}
\newcommand{\@starredsection}[1]{\oldsection*{\uline{#1}}}
\newcommand{\@unstarredsection}[1]{\oldsection{\uline{#1}}}

\let\oldsubsection\subsection
\renewcommand{\subsection}{\@ifstar{\@starredsubsection}{\@unstarredsubsection}}
\newcommand{\@starredsubsection}[1]{\oldsubsection*{\uline{#1}}}
\newcommand{\@unstarredsubsection}[1]{\oldsubsection{\uline{#1}}}
\makeatother

\makeatletter
\newtheoremstyle{underlinedthm}% name
  {6pt}% Space above
  {6pt}% Space below
  {\itshape}% Body font (Italic for theorems/lemmas)
  {}% Indent amount
  {\bfseries\itshape\ultrabold\color{darkblue}}% Theorem head font (Bold, Italic, Blue)
  {.}% Punctuation after theorem head
  {.5em}% Space after theorem head
  {\uline{\thmname{#1}\thmnumber{ #2}\thmnote{ (#3)}}}% CUSTOM UNDERLINED HEAD SPEC

\newtheoremstyle{underlineddef}% name
  {6pt}% Space above
  {6pt}% Space below
  {\normalfont}% Body font (Normal for definitions/remarks)
  {}% Indent amount
  {\bfseries\itshape\ultrabold\color{darkblue}}% Theorem head font (Bold, Italic, Blue)
  {.}% Punctuation after theorem head
  {.5em}% Space after theorem head
  {\uline{\thmname{#1}\thmnumber{ #2}\thmnote{ (#3)}}}% CUSTOM UNDERLINED HEAD SPEC
\makeatother

\theoremstyle{underlinedthm} % Applies italic body + underlined blue head
\makeatletter
\def\@settitle{\begin{center}%
  \baselineskip14\p@\relax
  \huge\bfseries\itshape\ultrabold\color{darkblue}%
  \expandafter\uline\expandafter{\@title}% <-- The magic fix
  \end{center}%
}
\makeatother

\newtheorem{theorem}{Theorem}[]
\newtheorem{definition}{Definition}[]
\newtheorem{discussion}{Discussion}[]

\newtheorem{lemma}{Lemma}
\newtheorem{proposition}{Proposition}
\newtheorem{notation}[definition]{Notation}
\newtheorem{corollary}{Corollary}

\theoremstyle{remark}

\newcommand{\ncp}[1]{\langle#1\rangle}

\newcommand{\Q}{\mathbb{Q}}

\newcommand{\Z}{\mathbb{Z}}

\newcommand{\Zbar}{\overline{\mathbb{Z}}}

\newcommand{\K}{\mathbb{K}}

\renewcommand{\P}{\mathbb{P}}
\newcommand{\A}{\mathbb{A}}

\newcommand{\p}{\mathfrak{p}}
\newcommand{\q}{\mathfrak{q}}
\renewcommand{\O}{\mathcal{O}}
\newcommand{\OK}{\mathcal{O}_{\mathbb{K}}}
\newcommand{\nm}{\mathfrak{Nm}}
\newcommand{\F}{\mathbb{F}}

\begin{document}

\title[Fundamental Theorem of Orders]{Unique decomposition of orders}
\author[G.D.Patil]{Gaurav Digambar Patil}
\address{LMSI address}
\email{gaurav231992@gmail.com}
\href{https://orcid.org/0009-0008-7073-0160}{orcid.org/0009-0008-7073-0160}
\date{\today}

\maketitle

\begin{abstract}
    We establish a Fundamental Theorem of Orders (FTO), which allows us to express any order (in a number field) uniquely as an intersection of \textit{irreducible orders}. Along this decomposition, the index (in the ring of integers) distributes multiplicatively, and the conductor factors into pairwise co-prime ideals.

    We use it to show a more general version of Furtwangler criterion about the structure of conductors of orders over $\Z$, as this answers a wide variations of such questions. In a future work, we will also give applications to weighted enumeration of number fields.
\end{abstract}

\section{Introduction}
The study of orders in algebraic number fields is a classical subject with deep connections to arithmetic geometry, algorithmic number theory, and arithmetic statistics. Given a number field $\K$ with its maximal order (the ring of integers) $\OK$, a general order $\O$ is a full-rank subring of $\OK$. While $\OK$ is a Dedekind domain, possessing unique prime ideal factorization and well-behaved localizations, general orders often exhibit complex local singularities (primes at which the localization is not a DVR). These deviations from being integrally closed are typically encoded by two highly interdependent invariants: the index $[\OK:\O]$ and the conductor ideal $c_\O$. 

Historically, controlling the local structure of non-maximal orders has been a significant hurdle in the enumeration of number fields and the evaluation of local factors of $L$-functions. In this paper, we establish a structural decomposition for arbitrary orders, which we term the \textit{Fundamental Theorem of Orders} (FTO). The FTO demonstrates that any order $\O$ can be expressed uniquely as a finite intersection of \textit{irreducible orders}—defined here as integral domains whose localizations fail to be Discrete Valuation Rings (DVRs) at exactly one prime ideal.

The arithmetic power of this decomposition lies in its separation properties. We prove that along this unique intersection, the index $[\OK : \O]$ distributes multiplicatively, and the conductor $c_\O$ factors into pairwise coprime ideals. This allows us to cleanly isolate non-unramified non-ramified (complex singularities) local behaviour of orders. 

While the primary motivation for developing this framework is its application to the weighted enumeration of number fields (which will be explored in a subsequent paper), the local control established by the FTO has immediate, standalone applications. To demonstrate this, we utilize the FTO to generalize Furtw\"{a}ngler's criterion, providing a complete structural classification of conductor ideals for arbitrary $R$-algebras.

\textbf{Outline of the paper.} In Section 2, we review and reframe standard results concerning Dedekind domains and Ostrowski's theorem to establish our local-to-global dictionary. In Section 3, we formally define irreducible orders, prove the associated Separation Lemmata, and establish the Fundamental Theorem of Orders. In Section 4, we relate the relative norms, ramification indices, and inertial degrees across these localizations, concluding with the application of our framework to generalize Furtw\"{a}ngler's criterion for conductors.

\section{Dedekind Domains in number-fields}
In this section, we show two minor variations on standard results:
\begin{itemize}
    \item First, we show is that the integral closure of any sub-ring of a number-field is a Dedekind domain. This is a minor adaptation of the standard algebraic number theory result stating that the integral closure of a Dedekind domain in a finite degree extension of its field of fractions is itself a Dedekind domain. While we expect this result exists in the standard literature, we provide a direct proof here for completeness.
    \item Second, we show a well-known but rarely explicitly written result in algebraic number theory/commutative algebra: the natural association mapping Dedekind domains (whose fractional field is a fixed number-field) to subsets of non-archimedean valuations on that number-field is bijective. This serves as a rewrite of Ostrowski's theorem. The outline is as follows: Ostrowski's theorem gives us a natural way to understand all possible non-archimedean valuations on a number-field via non-zero prime ideals of its ring of integers. Since every non-zero prime ideal in a Dedekind domain will induce a valuation, prime ideals in a Dedekind domain will carve out an explicit subset of all non-archimedean valuations. Now, standard commutative algebra dictates that any integral domain is completely determined by its localizations at maximal ideals. This tells us that such a map is in fact injective. On the other hand, standard theorems on the localization of rings applied to this setting tell us the association is surjective.
\end{itemize}

\begin{notation}
    Given a number-field, $\K,$ let $\OK$ be its ring of integers.
Let $M(\K)$ denote the set of non-Archimedean valuations of $\K$.
By Ostrowski's theorem, this set can naturally be identified with the set of non-zero prime ideals of $\OK$.
For any integral domain $R,$ we denote the integral closure of $R$ in its field of fractions by $\overline{R}$.
We denote the set of all prime numbers in $\Z$ by $\P_\Z.$
\end{notation}

\begin{lemma}\label{lem1-noeth} If $R$ is an integral domain such that $Frac(R)$ is a number-field $\K$, then
    \[
     \textit{ for any } p\in \P_\Z, \textit{ we have } |\sfrac{R}{pR}|\le p^{[\K:\Q]}.
\]
\end{lemma}
\begin{proof}
   Note that $\sfrac{R}{pR}$ is a $\F_p$-vector space ($\sfrac{\Z}{p\Z}$-module).We denote the canonical map from $R$ to $\sfrac{R}{pR}$ by $x\mapsto x^*.$
   
   If $x_i$'s were to satisfy a $\Z$-linear relationship, we may rewrite the relationship with coefficients whose greatest common factor is $1$.
Passing to $\sfrac{R}{pR}$ through $^*$, we get that $x_1^*,x_2^*,\cdots x_k^*$ are linearly independent.
Thus, if $x_1^*,x_2^*,\cdots x_k^*$ are linearly independent over $\F_p,$ then $x_1,x_2,\cdots,x_k$ are linearly independent over $\Z$ (or equivalently, over $\Q$).
It follows that 
   \[
   \dim_{\F_p}(\sfrac{R}{pR})\le [\K:\Q], \textit{ that is } |\sfrac{R}{pR}|\le p^{[\K:\Q]}.\qedhere
   \]
\end{proof}
\begin{lemma}\label{lem2-noeth}If $R$ is an integral domain such that $|\sfrac{R}{pR}|$ is finite, then
    \[
    \forall p\in \P_\Z: |\sfrac{R}{p^kR}|=|\sfrac{R}{pR}|^k.
\]
\end{lemma}
\begin{proof}
    Note that multiplication by $p^m$ from $R$ to $p^mR$ and from $pR$ to $p^{m+1} R$ are bijections and thus, the canonically induced map, $\sfrac{R}{pR}\longrightarrow \sfrac{p^mR}{p^{m+1}R}$ is also a bijection.
This gives us $|\sfrac{p^mR}{p^{m+1}R}|=|\frac{R}{pR}|.$ The lemma follows from the fact $|\sfrac{R}{p^kR}|=\sum_{i=0}^{k-1}|\sfrac{p^i R}{p^{i+1}R}|.$
\end{proof}
\begin{discussion}\label{R/I-finite}
    It follows that, if $R$ is a subring of a number-field, then $\sfrac{R}{nR}$ is a finite module for any integer $n.$

    Thus, given a non-zero proper ideal of $R,$ say $I,$ one may take any integral element $\alpha$ in said ideal and note that the norm of $\alpha$ is inside $\alpha\Z[\alpha]\in \alpha R$.
Thus, we have the following embeddings.
    \[
    I/(\alpha) \hookrightarrow R/(\alpha)\hookrightarrow  R/(\nm(\alpha)R), \textit{ and } R/I\hookrightarrow R/(\nm(\alpha)R).
\]
    Thus, both $I/(\alpha)$ and $R/I$ are finite. Thus, the following two propositions follow.
\begin{proposition}
        If $R$ is a sub-ring of a number-field $\K$, then it is Noetherian and has Krull dimension one (that is, every non-zero prime ideal is maximal as it generates a finite integral domain.).
\end{proposition}
    \begin{proposition}
        If $R$ is a sub-ring of a number-field $\K$, then $\overline{R}$ is a Dedekind domain.
\end{proposition}
\end{discussion}
\begin{notation}
    Given an integral domain $R,$ we define 
        \[
        M(R):=\{\p\subseteq R: \p \textit{ is a non-zero prime ideal in } R\}.
    \]
\end{notation}
    
\begin{notation}
        If $I$ and $J$ are $\Z$-sub-modules of $\K,$ we denote by $I\cdot J$ as the $\Z$-sub-module of $\K$ that is generated by products $i\cdot j$ where $i\in I$ and $j\in J.$   
\end{notation}

Since we will use these standard facts from commutative algebra repeatedly, we will write it here.
\begin{theorem}\label{localintersection} If $R$ is an integral domain, then
    \[
    R=\underset{\rho\in M(R)}{\cap} R_\rho.
\]
\end{theorem}
\begin{theorem}\label{localizationofintegralclosure}
        If $J$ is some multiplicatively closed subset of an integral domain $R$, then 
    \[
    J^{-1} \overline{R} = \overline{J^{-1} R}.
\]
\end{theorem}
See \cite[\href{https://stacks.math.columbia.edu/tag/0307}{Lemma 0307}]{stacks-project}
\begin{discussion}
    If $S$ is a Dedekind domain, whose fractional field is $\K,$ a number-field, then $\OK\subseteq S$, as $S$ is integrally closed and thus must contain the integral closure of $\Z$ in $\K,$ which is $\OK.$ Thus, given a non-zero prime ideal in $S$, say $\p$, then $\varrho:=\p\cap\OK$ is naturally a non-zero prime ideal in $\OK.$  
    
    Since $S$ is a Dedekind domain, $S_{\p}$ is a Discrete Valuation Ring (we will call this DVR from now on) and since $(\OK\backslash \varrho) \subseteq (S\backslash \p),$ we get $(\OK)_{\varrho}\subseteq 
S_{\p}.$ Now, Ostrowski's theorem tells us that the non-archimedean metric induced by $\varrho$ and $\p$ on $\K$ must be  $|\cdot|_\varrho$ (the metric induced by $\p$ when restricted to $\OK$ corresponds to $|\cdot|_\varrho$).
And thus, 
    \[
    \{x\in \K:|x|_\varrho\le 1\} = (\OK)_{\varrho} = S_{\p}.
\] 
    Furthermore, a natural consequence of Ostrowski's is that distinct non-archimedean metrics correspond to distinct primes,(as distinct primes naturally induce distinct metrics).
Thus, we may conclude the following. 
    \begin{proposition}
        If $S$ is a Dedekind domain with fractional field as $\K,$ a number-field, then $M(S)$ naturally injects into $M(\K)$.
\end{proposition}
    Furthermore, since any integral domain is the intersection of its localizations at its non-zero prime ideals (\cref{localintersection}), a Dedekind domain, with fractional field as $\K,$ a number-field, is completely determined by the image of $M(S)$ into $M(\K).$
    \begin{notation}    
        By abuse of notation, we treat $M(S)$ as a subset of $M(\K)$ whenever $S$ is a Dedekind domain whose fractional field, $\K,$ is a number-field.
\end{notation}
     Another consequence is:
    \begin{proposition}\label{ROK}
        If $R$ is a sub-ring of  a number-field, $\K,$ then $\overline{R}=R\cdot \OK.$ In particular, $S$ a subring of $\K,$ with fractional field $\K,$ is a Dedekind domain, if and only if $\OK\subseteq S.$
    \end{proposition}
    This is because of two facts: one being $\OK\subseteq R\cdot \OK \subseteq \overline{R}$ as integral closure of $\Z$ in $\K$ must be contained in $\overline{R}$ by definition.
The second is that localization of $R\cdot \OK$ at some prime will contain the localization of corresponding prime in $\OK,$ which by similar argument around Ostrowski's as above, will mean that it is equal to the corresponding localization of $\OK$ at corresponding prime ideal, which is a DVR.
Thus, $R\cdot\OK$ must be integrally closed. \qed
    
    Going back to subsets of $M(\K)$ and Dedekind domains with fractional field $\K$, if $M$ is any subset of $M(\K),$ then we may define $J_M:=\{x\in\OK:\forall v\in M, \textit{ we have  } x\notin v\}.$ Then $J_M^{-1}\OK$ is a Dedekind domain as it contains $\OK$ and properties of localization tells us that $M(J_M^{-1}\OK)=M.$ 
    We thus get:
    \begin{theorem}
        Given a number-field $\K,$ we have the following bijection:
    \begin{align*}
    \bigg\{\substack{\textit{Dedekind Domains}\\\textit{ with fractional field $\K$}}\bigg\} &\longleftrightarrow\bigg\{\textit{ Subsets of } M(\K)\bigg\}\\
        S&\longmapsto M(S),
    \end{align*}
    such that $S_1\subseteq S_2 \iff M(S_2)\subseteq M(S_1).$ Note that $J_{M(S)}^{-1}\OK=S$ and $M(J_{M}^{-1}\OK)=M.$
    \end{theorem}
\end{discussion}

\begin{corollary}\label{Idealupdown}
    If $I$ is an ideal in some Dedekind Domain $S$ with $Frac(S)=\K$, then we get  $I\cap \OK$ is an ideal of $\OK$ with the same prime factorization as $I$ in $S.$ In particular, 
    $$(I\cap \OK) \cdot S =I.$$
    Furthermore, since localization of $S$ at every prime matches the localization of $\OK$ at the corresponding prime ideal, we can say that the norm of $I$ in $S$ is the same as the norm of $I\cap \OK$ in $\OK.$ It follows that  
    \[
    \sfrac{\OK}{I\cap \OK} \simeq \sfrac{S}{I}.
    \]
\end{corollary}

\section{Local Algebras to Local Orders:  Separation Lemmata}
The objective of this section is to establish finer control over the index of an order $\O$ in the ring of integers $\OK$ of a number field $\K$. Specifically, we show that the localizations of a given order multiplicatively decompose this index, noting that these localizations only contribute non-trivially when they fail to be Discrete Valuation Rings (DVRs). 

To formalize this structure, we consider orders that differ from being a DVR at exactly one prime ideal; we define these as \textit{irreducible orders}. By expressing these irreducible orders as the intersection of a local non-DVR integral domain (with fractional field $\K$) and $\OK$, and noting \cref{localintersection}, we can write any order as a unique, finite intersection of irreducible orders over which the index distributes multiplicatively. Because the prime ideals corresponding to non-DVR localizations are governed by the conductor, the conductor itself splits multiplicatively into pairwise co-prime factors along this decomposition.

Beyond its immediate application in an upcoming paper on weighted enumeration of number-fields, this decomposition has implications for the structure of the local factors of $L$-functions that count orders in $\OK$ by discriminant (or, equivalently, by index). This framework also provides the necessary vocabulary for generalizing Dedekind-Kummer type theorems for more general types of Algebraic-Geometrically parametrized rings. For context, it is useful to view the classical Dedekind-Kummer theorem as a result on monogenic rings (those naturally associated to elements in  $\A^1(\Zbar)$, and the theorem of Corso et al. \cite{Corso2005DecompositionOP} as its analogue for binary rings (those naturally associated to elements in $\P^1(\Zbar)$.). Actually, in this context a better way to recognize a ring is by its geometric singularities. However, certain geometric singularities are already captured by $\OK$ by its ramified primes, thus we set up $S(R)$ as the singularities in $\O$ that do not match those of $\OK.$ Total set of singularities of $R$ is thus identified by $S(R)\cup \{v\in M(R)\backslash S(R): v\textit{ ramifies in } R\}$. We view $\{v\in M(R)\backslash S(R): v\textit{ ramifies in } R\}$ as simple singularities as these behave like DVRs, and $S(R)$ as the remaining more complex singularities (which are resolvable into simple ones or non singularities by moving to the integral closure). 

The apparatus developed here tracks these irreducible orders, their non-DVR localizations, and their associated conductors. We achieve this by identifying the set of primes at which the localization of a given ring is not a DVR, while simultaneously tracking the corresponding primes in its integral closure. We begin by proving a slight variation of standard algebraic number theory results regarding conductors and orders, generalized here for arbitrary sub-rings. (We suspect this specific variation may already exist in the standard commutative algebra literature).

We begin with the following lemma.
\begin{lemma}\label{coveringlemma}
    If $R$ is an integral domain, then $M(\overline{R})$ can be seen as a cover of $M(R)$ via the map $\pi:v\longrightarrow v\cap R.$ In other words, $\pi$ is surjective.
\end{lemma}
Note that $v\cap R$ has to be a proper ideal as $1\notin R\cap v$ and has to be prime as it is the kernel of the sequence of maps that maps it into an integral domain, namely $R\hookrightarrow \overline{R}\rightarrow  \sfrac{\overline{R}}{v}.
$

The surjectivity of the above map is standard fact in commutative algebra, see 
\cite[\href{https://stacks.math.columbia.edu/tag/00GQ}{Tag 00GQ}]{stacks-project}.
\begin{definition}
    If $\rho\in M(R)$, we define 
    $$
    \overline{\rho}:=\pi^{-1}(\rho),
    $$
    where $\pi$ is the map in \cref{coveringlemma}.
\end{definition}
Again we note that $\{\overline{\rho}:\rho \in M(R)\}$ partitions $M(\overline{R})$. Putting it into the context of previous discussion we get the following proposition.
\begin{proposition}
    If $R$ is a ring whose fractional field is a number-field, $\K,$ then
    \[
    \overline{\rho}=\{v\in M(\K):R_\rho \subseteq (\OK)_v\}.
    \]
\end{proposition}
Since we wish to, in some sense, identify these rings via their non-DVR localizations, we define:
\begin{notation}
    We set 
    \[
    S(R):=\{\rho\in M(R): R_\rho \textit{ is not a DVR.}\}.
    \]
\end{notation}
Note that this set will always be finite. Again reinterpreting $S(R)$ via Ostrowski's theorem, when the fractional field of $R$ is a number-field, we get the following proposition.
\begin{proposition}
    If $R$ is an integral domain whose fractional field is a number-field, $\K,$ then $S(R)$ is finite and can be qualified by 
    \[
    S(R)=\{\rho\in M(R): \exists v\in M(\K), R_\rho \subsetneq (\OK)_v\}.
\]
\end{proposition}
Thus, for all $\rho\notin S(R),$ we have $|\overline{\rho}|=1.$ The converse is not necessarily true.
We recall the concept of conductor of an integral domain (from commutative algebra) and its relationship to localizations which are not integrally closed local rings.
\begin{notation}
    If $I,J$ are $\Z$-submodules of a field $\K$, we use the following notation.
\[
    (I:J)_\K:=\{x\in \K:xJ\subseteq I\}.
    \]
\end{notation}

\begin{definition}
    If $R$ is an integral domain with fractional field $\K,$ we define the conductor of $R$ by
    \[
    c_{R}:=(R:\overline{R})_\K=\{x\in \K:x\overline{R}\subseteq R\}.
\]
\end{definition}
Next we will go over some standard properties of $c_R.$
\begin{discussion}
    Note that $x\overline{R}\subseteq R \implies xR\subseteq R$ and $xR\subseteq R \implies x\in R.$ Thus, $c_{R}$ is a subset of $R$ and thus of $\overline{R}$ as well.
Furthermore, if $x\in c_R$ and $y\in \overline{R}$ then $xy\overline{R}\subseteq x\overline{R}\subseteq R \implies xy\in c_R.$ Thus, $c_R$ is also an $\overline{R}$-ideal and therefore also an $R$-ideal.
Further-furthermore, if $J$ is an $\overline{R}$-ideal that is also a $R$-ideal, then by definitional verification, $J\subseteq c_R.$ Also, since $c_R\subseteq R$ we note that:
    \begin{lemma} The set $c_R\backslash\{0\}$ is the set of common denominators for $\overline{R}$ in $R.$ That is, 
        \[
        c_R\backslash\{0\}=\{x\in R:\overline{R}\subseteq \frac{1}{x}R\}.
\]
    \end{lemma}
    If $\rho$ is a prime ideal in $R$ and $c_R\not\subseteq \rho$ then $\exists y\in R\backslash \rho$ and $y\in c_R.$ Thus, 
    \[
    \overline{R}\subseteq \frac{1}{y}R\subseteq (R\backslash \rho)^{-1}R \implies R_\rho =(R\backslash \rho)^{-1}R= (R\backslash\rho)^{-1}\overline{R}.
\] 
    This by \cref{localizationofintegralclosure} tells us $R_\rho$ is integrally closed.
On the other hand, if $R_\rho$ is integrally closed then, since $R_\rho= (R\backslash \rho)^{-1}R,$ we can apply \cref{localizationofintegralclosure} to get: 
    \[
    (R\backslash \rho)^{-1}R=\overline{(R\backslash \rho)^{-1}R}=(R\backslash \rho)^{-1}\overline{R} \supset \overline{R}.
\]
    If $\overline{R}$ is finitely generated module over $R,$ say $\overline{R}=R\ncp{x_1,x_2,\cdots, x_k},$ we may write $x_i=r_i/s_i$ where $s_i\in (R\backslash \rho)^{-1}.$ Let $s=\prod_{i=1}^k s_i.$ Then, clearly we have $s\overline{R}\subseteq R$ and thus by definitional verification, $s \in c_R \cap (R\backslash\rho)$.
Thus, $c_R\not\subseteq \rho.$ 
    
    Now if $Frac(R)$ is a number-field, then we know that (\cref{R/I-finite}) $\overline{R}/R\simeq (\overline{R}/c_R)/(R/c_R)$ is finite.
Compiling together we get the following proposition.
    \begin{proposition}\label{conductor}
    If $R$ is an integral domain, then 
    \begin{itemize}
        \item $c_R$ is an ideal of $R$ as well as $\overline{R}.$
        \item $c_R$ is the largest $\overline{R}$ ideal which is completely inside $R$;
$c_R$ is maximal among $\overline{R}$ ideals $I$ satisfying $I\cap R=I.$
        \item For any $\rho\in M(R),$ $R_\rho$ is integrally closed if $c_R\not\subseteq \rho.$ If $\overline{R}$ is a finitely generated module over $R,$ then $R_\rho$ is integrally closed if and only if $c_R\in \rho.$ If $Frac(R)$ is a number-field, then for any $\rho\in M(R),$ $R_\rho$ is a DVR if and only if $c_R\subseteq \rho.$
    \end{itemize}
    \end{proposition}
\end{discussion}
\begin{theorem}\label{conductorlocalization}
    If $S$ is a multiplicative subset of an integral domain $R,$ where $Frac(R)$ is a number-field (or if $\overline{R}$ is 
a finitely generated module over $R$), then
    \[
    S^{-1}c_R=c_{S^{-1}R}
    \]
\end{theorem}
\begin{proof}It is obvious that $S^{-1}c_R\subseteq c_{S^{-1}R}.$ 

To show the converse, we note that $\overline{R}$ is finitely generated $R$-module (by assumption in theorem
).
Let $\overline{R}=R\ncp{x_1,x_2,\cdots, x_k}.$

Since $\overline{S^{-1}R}=S^{-1}\overline{R}=S^{-1}R\ncp{x_1,\cdots ,x_k},$ these $x_i$ also span $\overline{S^{-1}R}$ over $S^{-1}R$ (See \cref{localizationofintegralclosure}).
Thus, $\tau\in c_{S^{-1}R}$ tells us $\tau \overline{S^{-1}R} \subseteq S^{-1}R$ which in turn implies that for any choice of $i$ between $1$ and $k,$ we get $\tau x_i\in S^{-1}R$

We write $\tau \cdot x_i =\frac{r_i}{s_i}$ where $s_i\in S \textit{ and } r_i\in R.$ Let $s=\prod_{i=1}^ks_i.$ Then $s\cdot \tau \cdot x_i =r_i \cdot (\prod_{j\neq i} s_j) \in R.$ Thus, $s\tau\overline{R}=s\tau R\ncp{x_1,\cdots,x_k}\subseteq R$ or in other words, $s\tau\in c_R,$ which implies $ \tau\in S^{-1}c_R.$
\end{proof}
\begin{theorem}\label{conductorOK} If $R$ denotes a sub-ring of a number-field $\K$ with $Frac(R)=\K,$ then the following are true.
\begin{itemize}
    \item $c_{R\cap\OK}=c_R\cap \OK.$
    \item $c_{R\cap \OK}\cdot R =c_R.$
    \item $\sfrac{R}{c_R}\simeq \sfrac{R\cap\OK}{c_{R\cap\OK}}$
    \item $\sfrac{\overline{R}}{R} \simeq \sfrac{\OK}{R\cap \OK}$
\end{itemize}
\end{theorem}
\begin{proof}We prove the statements in the order given.
\begin{itemize} 
    \item  If $x\in c_R\cap \OK,$ then $x\overline{R}\subseteq R$ and $x\in \OK.$ Since, $x\OK \subseteq \OK$ it follows that 
    \[
        x\OK=x(\overline{R}\cap \OK)=x \overline{R} \cap x\OK\subseteq R\cap \OK \implies x\in c_{R\cap \OK}.
\]
    On the other hand, if $x\in c_{R\cap \OK}$ then 
    \[
    x\OK \subseteq R\cap \OK \implies x\OK \cdot R\subseteq (R\cap \OK)\cdot R \subseteq R.
    \]
    Now \cref{ROK} tells us, $R\cdot \OK=\overline{R}$.
It follows that $x\in c_{R}$. Since, $x\in c_{R\cap \OK}\subseteq \OK$ it follows that $x\in c_R\cap \OK$ showing $c_{R\cap\OK}\subseteq c_R\cap \OK$
    \item Since $c_{R\cap \OK} = c_R \cap \OK,$ we can multiply by $R$ to get $c_{R\cap \OK}R = (c_R \cap \OK)\cdot R\subseteq c_R.$  Since $c_{R\cap \OK}\cdot \OK = c_{R\cap\OK},$ applying \cref{ROK} we get $c_{R\cap \OK}\cdot R=(c_{R\cap \OK}\cdot\OK) \cdot R=c_{R\cap \OK}\cdot\overline{R}.$
    Since $c_R$ is an $\overline{R}$-ideal, \cref{Idealupdown} tells us 
    $c_{R\cap \OK}\cdot R=c_{R\cap \OK}\cdot\OK \cdot R=c_{R\cap \OK}\cdot\overline{R}=c_R.$
    \item Since $c_{R\cap \OK}= c_R\cap \OK,$ we can apply \cref{Idealupdown} to 
get the natural isomorphism
    $\sfrac{\OK}{c_{R\cap\OK}}\simeq \sfrac{\overline{R}}{c_R}$ where $m+c_{R\cap\OK}\mapsto m+c_R$.
Clearly, this map sends elements of $R\cap \OK/c_{R\cap \OK}$ to the elements of $R/c_R.$ To show surjection, we note that for any choice of $x\in \overline{R},$ our natural isomorphism gives us  and element, $m\in \OK$ such that $m-x\in c_R$.
Thus, when $x\in R \subseteq \overline{R},$ we have corresponding $m\in \OK,$ such that $m-x\in c_R$ which is equivalent to $m\in x+c_R\subseteq R,$ and thus $m\in R\cap \OK.$ It follows that the classes corresponding to $R/c_R$ in $\overline{R}/c_R$ naturally correspond to classes in $R\cap \OK/c_{R\cap\OK}$ in $\OK/c_{R\cap \OK}.$
    \item Follows directly from previous part as 
    \[
    (\sfrac{\overline{R}}{R})\simeq \sfrac{(\overline{R}/c_R)}{(\sfrac{R}{c_R})}) \textit{ and } \sfrac{\OK}{R\cap \OK}\simeq \sfrac{(\OK/c_{R\cap \OK})}{(R\cap \OK /c_{R\cap\OK})},
    \] 
    which in turn is a direct implication of $c_R\subseteq R\subseteq \overline{R}$ for any integral domain $R$ (see \cref{conductor}).
\end{itemize}
\end{proof}
\begin{definition}
    We say an integral domain $R$ whose field of fractions is a number-field, is irreducible if $|S(R)|=1.$
    
    In other words, $R$ is irreducible if and only if $\exists!
\rho$, a maximal ideal in $R$ such that $\forall \rho'\neq \rho$ in $M(R),$ we have 
    $R_{\rho'}$ is a DVR.
We call $R$ {\em an irreducible order}, when $R$ is an order that is irreducible.
\end{definition}
\begin{lemma}\label{FirstDecomp}
    Given any $R,$ a sub-ring of $\K$ with $Frac(R)=\K,$ we have 
    \begin{equation}
        R= \cap_{\rho \in S(R)} (R_\rho \cap \overline{R}).
\end{equation}
\end{lemma}
\begin{proof}
    Follows directly from 
    \[
    R_\rho\subseteq \overline{R}_\rho=\cap_{v\in \overline{\rho}}\overline{R}_v,
    \]
    and \cref{localintersection} applied to $R$ and $\overline{R}.$
\end{proof}

\begin{lemma}[Separation Lemma 1]\label{SEPARATION LEMMA}
    If $\O$ is an order such that its field of fractions is the number-field $\K$ and suppose that $\O=\cap_{\rho \in S(\O)} (\O_\rho \cap \OK)$ 
    from \cref{FirstDecomp}.
Let $c_\rho$ denote the conductor of $\O_\rho \cap \OK.$
    
    Then, we have
    \begin{itemize}
        \item  
            \[
            [\OK:\O]=\prod_{\rho \in S(\O)}[\OK:(\O_\rho \cap \OK)].
            \]
        \item If $\rho,\rho'\in S(\O)$ are distinct, then
            \[
            c_\rho+c_{\rho'}=\OK
            \]
            and
            \[
            c_{R}=\prod_{\rho\in S(\O)} c_{\rho}
            \]
    \end{itemize}
\end{lemma}
\begin{proof}
If $\rho,\rho'\in M(\O)$ are distinct, we have $\rho+\rho'=\O.$ This implies that, for any primary ideals $\partial, \partial'$ satisfying $\sqrt{\partial}=\rho$ and $\sqrt{\partial'}=\rho',$ we have $\partial + \partial'=\O.$ In other words, we have
\[
\sfrac{\O}{\partial \partial'}\simeq \sfrac{\O}{\partial}\times \sfrac{\O}{\partial'}.
\]
Note that \cref{conductor} tells us that the prime ideals of $\O$ containing $c_{\O}$ are exactly the ideals in $S(\O).$ Thus, the primary decomposition of $c_{\O}$ looks like 
\[
c_{\O}=\prod_{\rho\in S(\O)}\kappa_\rho
\] 
where the radical ideal of $\kappa_\rho$ (in $\O$) is $\rho.$

It follows that 
\[
\sfrac{\O}{c_{\O}} \simeq \prod_{\rho\in S(\O)} \sfrac{\O}{\kappa_\rho}\simeq \prod_{\rho\in S(\O)} \sfrac{\O_\rho}{\kappa_\rho\O_\rho}.
\]

Let $J_\rho$ denote the set $\O\backslash\rho.$ If $\rho'\neq \rho$ then $\kappa_{\rho'}$ has non trivial intersection with $J_\rho$ as $\kappa_{\rho'}+\rho =\O.$  Thus, $J_\rho^{-1}\kappa_{\rho'}=J_\rho^{-1}\O=\O_\rho.$

Thus, 
\[
J_\rho^{-1}c_{\O}=J_\rho^{-1}\prod_{\rho'\in S(\O)}\kappa_{\rho'}=J_{\rho}^{-1}\kappa_\rho=\kappa_\rho\O_\rho.
\]
But now, applying \cref{conductorlocalization} we see that $J_\rho^{-1}c_{\O}=c_{\O_\rho}$ and thus using \cref{conductorOK} we see that $c_{\O_\rho}=\kappa_\rho\O_\rho$.
Again applying \cref{conductorOK} we see that $\sfrac{\O_\rho}{c_{\O_\rho}}\simeq \sfrac{(\O_\rho\cap\OK)}{(c_{\O_\rho\cap\OK})}\simeq\sfrac{(\O_\rho\cap\OK)}{c_\rho}$.

Thus we may write

\begin{equation}\label{**}
    \sfrac{\O}{c_{\O}}\simeq\prod_{\rho\in S(\O)} \sfrac{(\O_\rho\cap \OK)}{c_\rho}.
\end{equation}

Since, $\kappa_\rho\subseteq c_\rho$ and if $\rho\neq \rho'$ then $1\in \kappa_\rho+\kappa_{\rho'}\subseteq c_{\rho}+c_{\rho'}.$ Since, $c_\rho$ and $c_\rho'$ are ideals in $\OK,$ it follows that 
\[
c_\rho+c_\rho'=\OK.
\]
Furthermore, we clearly have $\kappa_\rho\OK \subseteq c_\rho$ and our previous discussion gives $\kappa_\rho \O_\rho=c_{\O_\rho}$ Thus, appealing to \cref{ROK} , it follows that 
\[
c_{\O_\rho}=\kappa_\rho\O_\rho\OK=(\kappa_\rho\OK) \overline{\O_\rho}
\]
Now appealing to \cref{Idealupdown}, we get 

\[
c_\rho=c_{\O_\rho}\cap\OK=((\kappa_\rho\OK) \overline{\O}_\rho) \cap \OK=\kappa_\rho\OK.
\]
It follows that 
\[
c_{\O}=\prod_{\rho\in S(\O)}\kappa_\rho\OK=\prod_{\rho\in S(\O)}c_\rho.
\] 
Thus, we also get 
\begin{equation}\label{*}
    \sfrac{\OK}{c_{\O}}\simeq \prod_{\rho\in S(\O)} \sfrac{ \OK}{c_\rho}.
\end{equation}

Now appealing to $(\sfrac{\overline{R}}{R})\simeq \sfrac{(\overline{R}/c_R)}{(\sfrac{R}{c_R})}$
and using  \cref{*} and \cref{**}, we get that 

\[
            [\OK:\O]=\prod_{\rho \in S(\O)}[\OK:(\O_\rho \cap \OK)].
\]
\end{proof}
\begin{lemma}[Separation Lemma 2]\label{SEPARATION LEMMA 2}
    If $\O=\O_1\cap\O_2$ such that $c_{\O_1}+c_{\O_2}=\OK,$ then $c_{\O}=c_{\O_1}c_{\O_2}$ and
    $S(\O_1)$ and $S(\O_2)$ can be naturally identified with disjoint subsets of $S(\O)$ such that 
    \begin{itemize}
        \item $S(\O_1)\sqcup S(\O_2)=S(\O)$
        \item $\rho\in S(\O_1)\subseteq S(\O)\iff  \O_\rho=(\O_1)_\rho$
        \item $\rho\in S(\O_2)\subseteq S(\O)\iff  \O_\rho=(\O_2)_\rho.$
    \end{itemize}
\end{lemma}
\begin{proof}
    Note that $c_{\O}=c_{\O_1} \cap c_{\O_2}$ follows directly from \cref{conductor}.
As $c_{\O_1}+c_{\O_2}=\OK,$ it follows that $c_{\O}=c_{\O_1} \cap c_{\O_2}=c_{\O_1}c_{\O_2}.$
    
    Since, $c_{\O_1}+c_{\O_2}=\OK$, we can find $r_1\in c_{\O_1}$ and $r_2\in c_{\O_2},$ such that $r_1+r_2=1$.
Fix a prime ideal $\rho \in S(\O_1)$ (that is, $c_{\O_1}\subseteq \rho)$.
Let $\tau=\rho\cap \O.$ 
    
    Clearly, $\tau$ is a prime ideal in $\O$ and $\O_\tau\subseteq (\O_1)_\rho.$
    
    If $s\in (\O_1)_\rho,$ then $s=q/t$ for some $q\in \O_1$ and $t\in \O_1\backslash \rho $.
We simply note that since $tr_2\in c_{\O_2}\subseteq \O_2$ and $tr_2=t-tr_1\in \O_1$ we get $tr_2\in \O.$ We further note that $tr_2 \notin \tau$ ($tr_2\in \tau\subseteq \rho$ would imply $r_2\in \rho$ and since $r_1\in c_{\O_1}$ we would get $1=r_1+r_2\in \rho.$ Similarly, $qr_2\in c_{\O_2}\subseteq \O_2$ and $qr_2=q-qr_1\in \O_1$ we get $qr_2\in \O.$ 

    Thus, we get $s=q/t=(qr_2)/(tr_2)\in \O_\tau,$ implying 
    \[
    \rho \in S(\O_1) \Rightarrow (\O_1)_\rho=\O_{\rho\cap\O}. 
    \]
    If $\rho,\rho'\in S(\O_1)$ such that $\rho\cap \O =\rho'\cap\O,$ then $(\O_1)_\rho=(\O_1)_{\rho'}$ and thus $\rho=\rho'.$

    Thus, the canonical map 
$S(\O_1)\longrightarrow S(\O)$ given by $\rho\longrightarrow \rho\cap \O$ is injective. Similarly we can identify $S(\O_2)$ with a subset of $S(\O).$ To show that $S(\O_1)$ and $S(\O_2)$ are disjoint, we only need to note that if some $\rho\in S(\O)$ lies in both $S(\O_1)$ and $S(\O_2),$ we would immediately have $r_1\in c_{\O_1} \subseteq \rho$  and $r_2\in c_{\O_2} \subseteq \rho.$ Forcing $1= r_1+r_2 \in \rho.$ Contradicting $\rho$ is a prime ideal.
Finally if $\rho\in S(\O)$ is a prime ideal then $r_1r_2\in \rho$ which implies $r_1\in \rho$ or $r_2\in \rho.$ Since $r_1+r_2=1$, exactly one of $r_1,r_2$ is in $\rho.$ If $r_i\in \rho$ then $c_{\O}+r_i\O_i\subseteq\rho\O_i$.
If $t\in c_{\O_i}$ then $t=tr_2 +tr_1\in c_{\O} +r\O_1.$ Thus, $c_{\O_i}\in \rho \O_i$.
It follows that images of the injective maps $S(\O_i)\longrightarrow S(\O)$ are disjoint and cover $S(\O).$
\end{proof}
\begin{discussion}\label{irreducibleprime}
    Note that the above lemma tells us that irreducible orders mimic Euclid's property defining prime numbers, that is
    
    If $\O$ is an irreducible order, and $\O_A, \O_B$ are orders with $c_{\O_A}+c_{\O_B}=\OK$, then 
    \[
    \O_A\cap\O_B \subseteq \O \Rightarrow \O_A\subseteq \O \textbf{ or } \O_B \subseteq \O.
\]
\end{discussion}
\begin{theorem}[Fundamental Theorem of Orders]\label{FTO}
    Every order $\O$ can be written as an intersection of irreducible orders in a unique way such that the conductors of the irreducible orders are pairwise co-prime.
    The index (and conductor) of $\O$ in $\OK$ will be the product of the indices (and conductors) of the irreducible orders in $\OK$ in the given decomposition. This representation is given by 
    \[
    \O=\cap_{\rho\in S(\O)}(\O_\rho\cap \OK)
    \]
    where the localization of the order for a prime in $S(\O)$ matches with the localization of the corresponding irreducible component, namely $(\O_\rho\cap \OK)$.
\end{theorem}
\begin{proof}
    The existence of such a representation (and its particular form) follows from \cref{SEPARATION LEMMA} and uniqueness follows from \cref{SEPARATION LEMMA 2} and \cref{irreducibleprime}.
\end{proof}

\section{The functions $ef$ and $f$ of an Order.}
One often encounters a certain local behaviour of rings and some corresponding invariants around said construction, for example Dedekind Kummer Theorem and \cite{Corso2005DecompositionOP} tell us that the primary ideal over $p\in M(\Z)$ behaves in a certain way, that is, has a certain norm which can be compared to the norm of the prime ideal and the norm of $p$ in the integral closure of the corresponding localization. Recall the following standard definitions.
\begin{notation}
    For any valuation $v\in M(\K)$, we denote the absolute ramification index and the inertial degree (which clearly should have been called anti-inertial degree) of $v$ by $e_v$ and $f_v,$ respectively. That is, if $v\cap \Z=(p)$ for some prime in $\Z,$ then $e_v$ and $f_v$ are defined by  
    \[
    p\OK=\prod_{p\in v} v^{e_v} \textit{ and } \nm_{\OK}(v)=p^{f_v}.
    \]

    If $\mathbb{J}$ is an intermediate number-field, $\Q\subseteq \mathbb{J} \subseteq \mathbb{K}$, we define the relative ramification index and relative inertial degree as $e_v/e_{v\cap J}$ and $f_v/f_{v\cap J},$ respectively.
We denote these by $e_{v}(\K/\mathbb{J})$ and $f_{v}(\K/\mathbb{J})$, respectively.
\end{notation}
In light of the above set up, we make the following definitions.
\begin{definition}
    For any integral domain $R,$ whose fractional-field is the number-field $\K,$ and any non-zero prime ideal of $\rho\in M(R),$ we define 
    \[
    ef(\rho):=\sum_{v\in \overline{\rho}}e_v \cdot f_v.
    \] 
    Let $\rho\cap \Z=(p)$.
    We then define $f(\rho)$ by $|\O/\rho|=p^{f(\rho)}$.
We define $e(\rho)=\frac{ef(\rho)}{f(\rho)}.$
\end{definition}
\begin{definition}\label{ef-definition}
    If $\O$ is an irreducible order with $S(\O)=\{\rho\},$ we define 
    \[
    ef(\O):=ef(\rho), f(\O)=f(\rho) \textit{ and } e(\O):=e(\rho).
\]
\end{definition}
\begin{discussion}\label{Useofef(rho)}
    Let $R$ denote an integral domain whose fractional field is the number-field $\K.$
    For a prime ideal $\rho\in M(R)$ or any ideal in $\O$ containing a prime $p,$ $ef(\rho)$ clearly computes the $\Z_p$ rank of $\varprojlim \sfrac{R}{\rho^k}$ and hence does not quite care about the local ring or the semi-local ring represented in $\prod_{v\in \overline{\rho}}\O_{\K,v}$, where $\O_{\K,v}$ denotes the local completion of $\overline{R}$ at $v$. We note that since $\varprojlim \sfrac{R}{\rho^k}$ is a sub-module of a free $\Z_p$-module, structure theorem for $\Z_p$ modules tells us that 
    \[
    ef(\rho)=rank_{\Z_p}(\varprojlim \sfrac{R}{\rho^k})=\dim_{\F_p}((\varprojlim \sfrac{R}{\rho^k})/(p))=\dim_{\F_p}(\sfrac{R_\rho}{(pR_\rho)}).
    \]
    Thus, $\nm(pR_\rho)=p^{ef(\rho)}$, or more specifically, if the primary component of $pR$ given by $\q_\rho$ that sits over $\rho\in M(R)$ in the primary decomposition of $pR$ satisfies $\nm(\q_\rho)=p^{ef(\rho)}$.

\end{discussion}
\begin{lemma}\label{Normefrelation}
    If $p$ is a prime over $\Z$, has primary decomposition $\prod_i q_i=\cap_i q_i$ in some integral domain $R$, whose fractional field is the number-field $\K,$ where $\q_i$ is $\rho_i$-primary ideal then $\nm(q_i)=p^{ef(\rho_i)}.$
\end{lemma}

\begin{discussion}
    Suppose $\O$ is an irreducible order (with fractional field $\K$) with $S(\O)=\{\rho\}$ such that $\overline{\rho}=\{v_1,v_2,\cdots,v_k\}\subseteq M(\K).$ Let $p$ be a prime number such that $(p)=\rho\cap \Z.$ If the conductor of $\O$ is $c_{\O},$ then we know that 
    \[
    c_{\O}=(v_1)^{a_{v_1}}(v_2)^{a_{v_2}}\cdots (v_k)^{a_{v_k}}
    \]
    for some $a_i\in \Z,$ where $a_i\ge 1$ and $v_i \cap \Z=(p).$ It follows from \cref{Normefrelation} that $ef(\O)=\sum_{i=1}^k e_{v_i}f_{v_{i}}.$

    Furthermore, since $c_{\O}\subseteq \O \subseteq \OK$ we see that $\O$ is completely determined by $\sfrac{\O}{c_{\O}},$ seen as a sub-ring of $\sfrac{\OK}{c_{\O}}.$  This naturally sets up a bijection between local sub-rings of $\prod_{v\in S}\O_{\K,v}$ with irreducible orders $\O\subsetneq \OK$ satisfying $\sqrt{c_{\O}}=\prod_{v\in S}v$ given by 
    \[ 
    \O\longrightarrow \varprojlim \sfrac{\O}{\rho^k}.
    \]
    The inverse of this map is given by intersecting the given local ring in $\prod_{v\in S}\O_{\K,v}$ with $\OK$ as it sits diagonally in $\prod_{v\in S}\O_{\K,v}$. We thus get the following proposition.
\end{discussion}
\begin{proposition}[Bijection:Irreducible orders-Local algebras in p-adic extensions]
    There is a natural bijection between local sub-rings of $\prod_{v\in S}\O_{\K,v}$ and irreducible orders $\O\subsetneq \OK$ satisfying $\sqrt{c_{\O}}=\prod_{v\in S}v$ given by 
    \[ 
    \O\longrightarrow \varprojlim \sfrac{\O}{\rho^k}.
    \]
    The inverse of this map is given by intersecting the given local ring in $\prod_{v\in S}\O_{\K,v}$ with $\OK$ as it sits diagonally in $\prod_{v\in S}\O_{\K,v}$. We thus get the following proposition.
\end{proposition}
The above set up gives us a nice simple way to get the Furtwangler's condition for an ideal to be a conductor ideal.
Consider an ideal $c$ of $\O_\K$ such that $c\cap \Z = (p^{m+1})$ for some prime $p$ in $\Z.$ If $\O$ is an irreducible order with conductor $c,$ then $p^m\Z+c$ cannot be an ideal of $\OK.$ If it were then this would violate the conductor being the largest ideal of $\OK$ sitting in $\O.$ Set $\tau=p^m\OK+c$.
Thus, $c$ cannot satisfy, $p^m\Z+c=p^m\OK+c.$  In other words, we cannot have 
\begin{align*}
    &\forall x\in \OK\:  \exists \:r\in  \Z : (x-r)p^m\subseteq c\\
    &\iff \forall x\in \OK\:  \exists \:r\in  \Z : (x-r)p^m\tau^{-1}\subseteq c \tau^{-1}\\
    &\iff \forall x\in \OK\:  \exists \:r\in  \Z : (x-r)\subseteq c \tau^{-1}\\
    &\iff c\tau^{-1} \textit{ is a prime ideal $v$ with } f_v=e_v=1 
\end{align*} 
Unpacking the final condition as a condition about the prime factorization of $c$ gives us the Furtwangler's condition.
The only if part is trivial and follows by considering the order $\Z+c$ which will be an order with conductor $c$ when $c$ satisfies the Furtwangler's condition: 
\begin{theorem} 
    Suppose the following holds. 
    An ideal $c$ corresponds to a conductor of some order $\O$ in $\OK,$ if and only if one of the following conditions hold.
    \begin{itemize}
        \item For all $v|c,$ we have either $f_{v}\ge 2$ or $v(c)\not\equiv 1\bmod e_{v}$ (Here, $v(c)$ denotes the power of $v$ dividing $c.$).
        \item If $f_{v}=1$ and $v(c)\equiv 1\bmod e_{v}$ for some valuation $v|c,$ then there must exist another valuation $w$ such that
        $\frac{w(c)}{e_{w}}> \frac{v(c)-1}{e_{v}}$ and $w\cap R=v\cap R =\rho.$
    \end{itemize}
\end{theorem}
In fact, combining with Fundamental Theorem of Orders or \cref{FTO} it gives a slightly more nuanced answer about orders and associated conductors and what can exist. It tells us that Furtwangler's condition has to hold for irreducible components and not just the order itself.

For example, it tells us that there cannot exist an order $\O$ with conductor $vw$ with $f_v=e_v=f_w=e_w=1$ and $\sfrac{\O}{vw}\simeq (\sfrac{\Z}{p\Z})^2.$ While this one is obvious, the author is not sure all such conditions can be considered obvious.
Furthermore we can talk about orders that are some sub-ring $R$-modules (not necessarily full rank) and even get a Furtwangler's equivalent condition for conductor ideals in $R$-algebras in terms of prime factorizations of prime ideals in $R$ can the powers and inertial and ramification degrees of primes sitting over them.
To be specific, using the exact same argument as above, we get the following generalization, \cite{furtwangler1919fuhrer}, also see \cite{Reinhart2015ANO}.
\begin{theorem}(Classification of conductor ideals of $R$-algebras whose fractional field is some fixed number-field $\K.$)
    For each prime $\rho$ of $R$, suppose 
    \[
    \rho\OK= \prod_{i=1}^r v^{e_{R,v}} \textit{ with } f_{R,v}:=\log(\nm_{\OK}(v)) / \log(\nm_R(\rho)).
    \]
    
    Then, the ideal $c$ (an ideal of $\OK$) is a conductor ideal of an order $\O$ which is an $R$-algebra if and only if one of the following condition holds.
\begin{itemize}
        \item For all $v|c,$ we have either $f_{R,v}\ge 2$ or $v(c)\not\equiv 1\bmod e_{R,v}$ (Here, $v(c)$ denotes the power of $v$ dividing $c.$).
        \item If $f_{R,v}=1$ and $v(c)\equiv 1\bmod e_{R,v}$ for some valuation $v|c,$ then there must exist another valuation $w$ such that
        $\frac{w(c)}{e_{R,w}}> \frac{v(c)-1}{e_{R,v}}$ and $w\cap R=v\cap R =\rho.$
\end{itemize}
\end{theorem}
Another consequence is 
\bibliographystyle{abbrv}
\bibliography{main}
\end{document}